\documentclass[twoside]{article}
\usepackage{amsmath,amsthm,amssymb}
\usepackage{url}
\usepackage[spanish,english]{babel}
\newtheorem{theorem}{Theorem} 
\newtheorem{lemma}[theorem]{Lemma}     
\newtheorem{corollary}[theorem]{Corollary}

\newtheorem{definition}{Definition}
\newtheorem{question}{Question}

\begin{document}
\label{begin-art}

\begin{center}
{\LARGE\bfseries Full operators and algebras generated by invertible operators \par }

\vspace{5mm}
{\large Wilson R. Pacheco R (\url{wpacheco@luz.edu.ve})}\\[1mm]
Departamento de Matem\'{a}ticas, FEC\\ Universidad del Zulia\\
   Maracaibo, Venezuela
\end{center}

\begin{abstract}
In the present work we characterize full operators and show some properties for bounded below nonfull operators. Under the results developed for full operators, we affirmatively respond two questions formulated by Bravo and Feintuch about algebras generated by invertible operators.


\end{abstract}

\section{Introduction}
\label{intro}
    It is a well known fact from linear algebra that if $T$ is an
invertible operator on a finite dimensional space, then $T^{-1}$ is a
polynomial in $T$. This fact is false if the vector space is infinite
dimensional. More over, there are examples in which $T^{-1}$ is not even the
limit of polynomials in $T$. The bilateral shift operator in $l_{2}(Z)$ is a
good example of this fact.

We will denote by $\mathtt{A}_{T}$ the weakly closed algebra generated by $T$
and the identity operator. If $T^{-1}$ belongs to $\mathtt{A}_{T}$, then $T^{-1}$ can be weakly approximated by polynomials in $T$. It follows that any invariant subspace for $T$ is also invariant for $T^{-1}$. The problem of determining when  $T^{-1}$ belongs to $\mathtt{A}_{T}$ has been studied for several authors (see \cite{Bravo}, \cite{Erdos1}, \cite{Erdos2}, \cite{Feintuch1}, \cite{Feintuch2}, \cite{Feintuch3}, \cite{Ortuñez} ).

 If $\mathrm{lat}T$ denotes the lattice of the invariant subspaces  for $T$,
then the previous result assures that $TM=M$ for all $M\in \mathrm{lat}T$.

An operator that satisfies $\overline{TM}=M$ for all $M\in \mathrm{lat}T$%
, is called a full operator. Therefore, if $T^{-1}$ belongs to $\mathtt{A}%
_{T} $, then $T$ is necessarily full.

In \cite{Bravo} Bravo studies conditions under which $\ T^{-1}$ belongs to the
weak algebra generated by $T$ and the identity operator. He also
characterized the full operators among other things. In his work the following conjecture is stated:

\begin{question}
If $T$ is an invertible operator and $\mathtt{A}_{T}$ contains an injective
quasinilpotent operator $R$ then $T^{-1}\in \mathtt{A}_{T}$.
\end{question}

In \cite{Feintuch2} Feintuch made a similar conjecture, with the quasinilpotent condition
on $R$ replaced by compactness.

In this work we characterize full operators and show some properties
for bounded below nonfull operators. Based on our results developed for
full operators we prove both Bravo and Feintuch's conjectures.

Let $H$ be a separable Hilbert space, and $L(H)$ be the algebra of all bounded
operators on $H$. We denote by $\mathrm{lat}T$ the lattice of all
invariant subspaces for $T$, i.e, $\mathrm{lat}T=\{M:M$ is closed and$\
TM\subset M\}$.

\begin{definition}
An operator $T\in L(H)$ is called a full operator if $\overline{TM}=M$ for
all $M\in \mathrm{lat}T$.
\end{definition}
Full operators were introduced by Erdos in \cite{Erdos1}.

In what follows, $\mathtt{A}_{T}$ will denote the weak algebra generated by $T$ and the
identity operator, $\mathrm{alglat}T=\left\{ S\in L(H):\mathrm{lat}T\subset \mathrm{lat}S\right\} $, and $\left\{T\right\} ^{\prime }$ is
the conmutant of $T$.

It is known that, if $T$ is invertible and $T^{\left( n\right) }$ is full for every $n\in N$, where $T^{\left( n\right) }$ is the direct sum of $n $
copies of $T$, then $T^{-1}\in $ $\mathtt{A}_{T}$ (see Corollary 1.2.3 in \cite{Bravo}).

Given $x\in H$, $T\in L(H)$ and $n\in N\cup \{0\}$, we denote by $M(n,x,T)$ the closure of the $T$ - cyclical subspace generated by $T^{n}x$.

\begin{lemma}
Let $T\in L(H)$ and let $M\in \mathrm{lat}\,T$ be such that $\overline{TM}\varsubsetneqq M$. Then there exists $x\in M$, with $\left\Vert x\right\Vert =1$,
such that $\ x\in M(1,x,T)^{\bot }$.

If $T$ is bounded below, then $T^{n-1}x\notin M(n,x,T)$. In particular, dim$\left(
M(n,x,T)\right) =\infty $ for all $n\in N$
\end{lemma}

\begin{proof}
Since $\overline{TM}\varsubsetneqq M$, we can find $x\in M\cap \overline{TM}^{\bot
}$  with $\left\Vert x\right\Vert =1$. Since $\ M(1,x,T)\subset \overline{TM}$, then
$\overline{TM}^{\bot }\subset M(1,x,T)^{\bot }$. Therefore, $x\in
M(1,x,T)^{\bot }$.

On the other hand, if $T$ is bounded below then $T^{n-1}x\notin \
M(n,x,T)$. Indeed, let's suppose that there exists $\left\{ p_{\alpha }\right\} $ such
that $T^{n-1}x\longleftarrow p_{\alpha }(T^{n})x=T^{n-1}q_{\alpha }(T)x$.
Then $T^{n-1}\left( x-q_{\alpha }(T)x\right) \longrightarrow 0$, and since $T$
is bounded below, we have $x-q_{\alpha }(T)x\longrightarrow 0$, which implies that $x\in M(1,x,T)$, contradicting the fact that $x\in M(1,x,T)^{\bot }.$
\end{proof}

\begin{theorem}
Let $T\in L(H)$. The following propositions are equivalent:

\begin{enumerate}
\item $T$ is full.

\item If $x\in M(1,x,T)^{\bot }$, then $x=0$.
\end{enumerate}
\end{theorem}

\begin{proof}
Let's suppose that (ii) is true and  $T$ is not full. Then, there exists $\ M\in \mathrm{lat}\,T$ such that $TM\varsubsetneqq M.$

Therefore, by lemma 1 there exists $x\in M,$ with $x\neq 0,$ such that $x\in M(1,x,T)^{\bot },$ which contradicts (ii).

Let's  suppose now that (i) is true and that there exists $x\neq 0$ such that $x\in M(1,x,T)^{\bot }.$ Then $x\notin M(1,x,T)=\overline{TM(0,x,T)}$. Therefore $\overline{TM(0,x,T)}\varsubsetneqq M(0,x,T)$, which contradicts the fact that $T$ is full.
\end{proof}

\begin{theorem}
Let $T$ be an injective operator and suppose that $\mathrm{alglat}\, T$
contains a quasinilpotent full operator $Q$. Then $T$ is full.
\end{theorem}

\begin{proof}
If that $T$ is not full, there exists $x\in H$, $x\neq 0$, and $\left\Vert x\right\Vert =1,$ such that $\ x\in M(1,x,T)^{\bot }$.

Since $\ M(0,x,T)\in \mathrm{lat}Q,$ we have that $Qx=\alpha x+y$, where $y\in M(1,x,T)$.

For each $n\in N$, we have that, $Q^{n}x=\alpha ^{n}x+y_{n}$, where $y_{n}\in M(1,x,T)$.

Then, $\left\vert \alpha \right\vert =\left\vert \left\langle \alpha
^{n}x+y_{n},x\right\rangle \right\vert ^{1/n}=\left\vert \left\langle
Q^{n}x,x\right\rangle \right\vert ^{1/n}\leq \left\Vert Q^{n}\right\Vert ^{\frac{1}{n}}$.

Since $Q$ is quasinilpotent, we have $\alpha =0$. Therefore $\left\langle
Q^{n}x,x\right\rangle =0$, and by theorem 4, it follows that $Q$ can not be full in
contradiction with the hypothesis.
\end{proof}

\begin{corollary}
Let $T$ be a invertible operator and suppose that $\mathrm{alglat}\,T\cap\left\{T\right\} ^\prime$ contains a compact full operator K. Then $T$ is full.
\end{corollary}

\begin{proof}

Suppose that $T$ is not full, and let $x\in H$ be such that $ x\in M(1,x,T)^{\bot }$ and $\left\Vert x\right\Vert =1$.

Let  $M=M(0,x,T)$, and notice that $T\left\vert _{M}\right. $ is not full. Let also $K_{1}=K\left\vert _{M}\right. $. Then $K_{1}$ is a compact full operator and $R_{1}\in \mathrm{alglat}\,T\left\vert _{M}\right. \cap\left\{ T\left\vert _{M}\right. \right\} ^\prime $.
Since in the proof on the main theorem in \cite{Feintuch2}, we may as well assume that $M$ contains no common invariant subspace of $T$ and $T^{-1}$. In particular, $T$ does not have finite dimensional
invariant subspaces in $M$

If $\sigma \left( K_{1}\right) =\left\{ 0\right\} $, then $K_{1}$  is
quasinilpotent and we can apply the previous theorem to guarantee that $T\left\vert_{M}\right. $ is full, contradicting the observation made on $T\left\vert_{M}\right. $.

In case that $\sigma \left( K_{1}\right) \neq\left\{ 0\right\} $, $ T$
would have an invariant finite dimensional subspace in $M$, in contradiction
to above assume.
\end{proof}

We now proceed to answer question 1

\begin{theorem}
Let $T\in L(H)$ be invertible and suppose that $\mathtt{A}_{T}$ contains an
injective quasinilpotent operator $Q$. Then $T^{-1}\in $ $\mathtt{A}_{T}$.
\end{theorem}

\begin{proof}
If $Q$ is an injective quasinilpotent operator in $\mathtt{A}_{T}$, then $
Q^{\left( n\right) }$ is an injective quasinilpotent operator in $\mathtt{A}_{T^{\left( n\right) }}.$ So it suffices to show that the assumptions of the theorem imply that $ T$ is full.

In order to do so, let us assume that $T$ is not full. Then we can find $x\in H$, with $\left\Vert x\right\Vert =1$, such that $\ x\in M(1,x,T)^{\bot }$. Since the above corollary we may as well assume that $M$ contains no common invariant subspace of $T$ and $T^{-1}$. In particular, we have $\bigcap\limits_{n\in N}T^{n}\left(M(0,x,T)\right) =\{0\}$.

Let $M=M(0,x,T)$. Since $Q$ is injective, $QM\neq \{0\}$. Let k be the largest natural number such
that $QM\subset T^{k}\left( M\right) $, and $x_{n}$ be a unitary vector of the
one dimensional space $T^{n}\left( M\right) \ominus T^{n+1}\left( M\right) $.

Now,
\begin{equation}
Tx_{n}=\alpha _{n}x_{n+1}+y_{n+2},
\end{equation}
where $y_{n+2}\in T^{n+2}\left( M\right) $.

By (1) we have
\begin{equation}
\left\vert \alpha _{n}\right\vert =\left\vert \left\langle \alpha
_{n}x_{n+1}+y_{n+2},x_{n+1}\right\rangle \right\vert =\left\vert
\left\langle Tx_{n},x_{n+1}\right\rangle \right\vert \leq \left\Vert
T\right\Vert .
\end{equation}

Similarly, from the equality
\begin{equation}
x_{n}=\alpha _{n}T^{-1}x_{n+1}+T^{-1}y_{n+2},
\end{equation}
we obtain
\begin{equation}
\left\vert \alpha _{n}\right\vert \geq \left\Vert T^{-1}\right\Vert ^{-1}.
\end{equation}
The equation (1) also implies that
\begin{equation}
T^{nk}x_{l}=\left( \alpha _{l}\alpha _{l+1}\cdots \alpha _{l+nk}\right)
x_{l+nk}+z_{l+nk+1},
\end{equation}
where $z_{l+nk+1}\in T^{l+nk+1}\left( M\right) $, for all $l\geq 0$, and $%
n,k\in N$.

On the other hand, $Qx_{0}=\beta _{0}x_{k}+w_{k+1}$, where $w_{k+1}\in
T^{k+1}\left( M\right) $. In general
\begin{equation}
Qx_{nk}=\beta _{n}x_{\left( n+1\right) k}+w_{\left( n+1\right) k+1},
\end{equation}
where $w_{\left( n+1\right) k+1}\in T^{\left( n+1\right) k+1}\left(
M\right) $, because $QT^{nk}(M)\subset T^{nk}Q(M)\subset T^{(n+1)k}(
M)$. By the choice of $k$, we have $\beta _{n}\neq 0$ for all $n$.
Then we can write

\[
\left\langle T^{nk}Qx_{0},x_{\left( n+1\right) k}\right\rangle =\left\langle
\beta _{0}T^{nk}x_{k}+T^{nk}w_{k+1},x_{\left( n+1\right) k}\right\rangle
=\beta _{0}\left( \alpha _{k}\alpha _{k+1}\cdots \alpha _{\left( n+1\right)
k}\right),
\]
and
\begin{eqnarray*}
\left\langle T^{nk}Qx_{0},x_{\left( n+1\right) k}\right\rangle
&=&\left\langle QT^{nk}x_{0},x_{\left( n+1\right) k}\right\rangle \\
&=&\left\langle \left( \alpha _{0}\alpha _{1}\cdots \alpha _{nk}\right)
Qx_{nk}+Qz_{nk+1},x_{\left( n+1\right) k}\right\rangle \\
&=&\beta _{n}\left( \alpha _{0}\alpha _{1}\cdots \alpha _{nk}\right) .
\end{eqnarray*}
Therefore, the inequality
\begin{eqnarray}
\left\vert \frac{\beta _{n}}{\beta _{0}}\right\vert &=&\left\vert \frac{\left( \alpha _{k}\alpha_{k+1}\cdots \alpha _{\left( n+1\right) k}\right) }{\left( \alpha_{0}\alpha _{1}\cdots \alpha _{nk}\right) } \right\vert \nonumber\\
&=&\left\vert \frac{\left(\alpha _{nk+1}\alpha _{nk+2}\cdots \alpha _{\left( n+1\right) k}\right) }{\left( \alpha _{0}\alpha _{1}\cdots \alpha _{k-1}\right) }\right\vert
\geq \frac{1}{m}
\end{eqnarray}
holds for all $n\geq 0$, where $m=\left( \left\Vert T^{-1}\right\Vert \left\Vert
T\right\Vert \right) ^{k}$.

By (6) we have $\left\Vert Q^{n}x_{0}\right\Vert \geq \left\vert \beta
_{0}\beta _{1}\cdots \beta _{n}\right\vert ,$ and since Q is quasinilpotent we have
\begin{equation}
\left\vert \beta _{0}\beta _{1}\cdots \beta _{n}\right\vert ^{\frac{1}{n}%
}\rightarrow 0.
\end{equation}
But equation (7) implies that
\[
\left\vert \beta _{0}\beta _{1}\cdots \beta _{n}\right\vert ^{\frac{1}{n}%
}\geq \frac{\left\vert \beta _{0}\right\vert }{m}0
\]
which clearly contradicts (8). Thus $T$ is necessarily full.
\end{proof}

\begin{corollary}
Let $T\in L(H)$ be an invertible operator and suppose that $\mathtt{A}_{T}$
contains an injective compact operator $K$. Then $T^{-1}\in $ $\mathtt{A}_{T}$.
\end{corollary}

\begin{proof}
If $K$ is a quasinilpotent operator the result follows from the application of the previous theorem. Otherwise
the proof can be argued in the same fashion as in corollary 4.
\end{proof}

The above corollary responds the question posed by Feintuch.

Theorem 5 are strongly motivated by the proof of theorem 1.2.15 in
\cite{Bravo}.

\label{end-art}
\end{document}